\documentclass[12pt]{amsart}

\usepackage[backend=biber,style=numeric,sorting=nyt,doi=false,url=false,isbn=false,giveninits=true]{biblatex}
\usepackage{hyperref}
\usepackage{tabularx,booktabs}
\usepackage{caption}
\usepackage{amsmath} 
\usepackage{bbm}
\usepackage{amsfonts}
\usepackage{amscd}
\usepackage{amsthm}
\usepackage{amssymb} 
\usepackage{cancel}
\usepackage{geometry}
\usepackage{latexsym}
\usepackage{eufrak}
\usepackage{euscript}
\usepackage{epsfig}
\usepackage{graphics}
\usepackage{array}
\usepackage{enumerate}
\usepackage{dsfont}
\usepackage{color}
\usepackage{wasysym}
\usepackage{multirow}
\usepackage{pdfsync}
\newcommand{\qe}{\end{equation}}

\newcommand{\R}{{\mathbb R}}
\newcommand{\N}{{\mathbb N}}
\newcommand{\Z}{{\mathbb Z}}
\newcommand{\C}{{\mathbb C}}

\newcommand{\Hmm}[1]{\leavevmode{\marginpar{\tiny%
$\hbox to 0mm{\hspace*{-0.5mm}$\leftarrow$\hss}%
\vcenter{\vrule depth 0.1mm height 0.1mm width \the\marginparwidth}%
\hbox to
0mm{\hss$\rightarrow$\hspace*{-0.5mm}}$\\\relax\raggedright #1}}}

\newtheorem{theorem}{Theorem}[section]

\newtheorem{lemma}[theorem]{Lemma}
\newtheorem{corollary}[theorem]{Corollary}

\newtheorem{remark}[theorem]{Remark}

\begin{document}

\title{Continuum limit of fourth-order Schr\"{o}dinger equations on the lattice}

\author{Jiawei Cheng}
\address{Jiawei Cheng: School of Mathematical Sciences, Fudan University, Shanghai 200433, China; Laboratoire de Math\'{e}matiques d'Orsay, CNRS, Universit\'{e} Paris-Saclay, F-91405 Orsay Cedex, France}
\email{\href{mailto:jiawei.cheng@universite-paris-saclay.fr}{jiawei.cheng@universite-paris-saclay.fr}}

\author{Bobo Hua}
\address{Bobo Hua: School of Mathematical Sciences, LMNS, Fudan University, Shanghai 200433, China; Shanghai Center for Mathematical Sciences, Fudan University, Shanghai 200433, China}
\email{\href{mailto:bobohua@fudan.edu.cn}{bobohua@fudan.edu.cn}}

\begin{abstract}
    In this paper, we consider the discrete fourth-order Schr\"{o}dinger equation on the lattice $h\mathbb{Z}^2$. Uniform Strichartz estimates are established by analyzing frequency localized oscillatory integrals with the method of stationary phase and applying Littlewood-Paley inequalities. As an application, we obtain the precise rate of $L^2$ convergence from the solutions of discrete semilinear equations to those of the corresponding equations on the Euclidean plane $\mathbb{R}^2$ in the contimuum limit $h \rightarrow 0$.
\end{abstract}

\maketitle

\section{Introduction}\label{sec-intro}

We consider the discrete homogeneous semilinear fourth-order Schr\"{o}dinger equation (DSFS, in short) with Cauchy data for $\lambda \in \R$, $p>0$,  
\begin{equation}\label{equ-DSFS}
    \left\{
    \begin{aligned}
        & i \partial_t u(x,t) + {\Delta}^2_h u(x,t)  =\lambda |u|^{p-1}u, \\
        & u(x,0) = f(x).
    \end{aligned}
    \right.
\end{equation}
Here the complex valued function $u(x,t)$ is defined on $h\Z^d \times \R$, with the stepsize $h \in (0,1]$ and $d \in \Z_+$, where the lattices are given by
\begin{equation*}
    h\Z^d := \{x=hn:n=(n_1,\cdots,n_d) \in \Z^d\}.
\end{equation*}
On such lattices, we define the discrete Laplacian, depending on $h$, by 
\begin{equation*}
    (\Delta_hu)(x) :=\sum_{j=1}^{d} h^{-2} \left(u(x+he_j)+u(x-he_j)-2u(x)\right),
\end{equation*}
where $\{e_j\}$ is the orthnormal basis on $\R^d$.

The DSFS is a specific example of the following discrete fractional Schr\"{o}dinger equation,
\begin{equation*}
    i \partial_t u + {(-\Delta_h)}^{\alpha} u  =\lambda |u|^{p-1}u,
\end{equation*}
where $(-\Delta_h)^{\alpha}$ is defined via its Fourier multiplier, see Section \ref{ssec-sobo}. The family of models appear naturally in physics and biology such as the dynamics of Bose-Einstein condensates and DNA, refer to \cite{CBF01,K09,MCG99}. When $\alpha=1$, it is the classical Schr\"{o}dinger equation, while for $\alpha=\frac{1}{2}$, it is closely related to the wave equation, see \cite{BCH23,BCH24,S98,SK05}.

When $\alpha=2$ and $h=1$, the first author has proved in \cite{C24} that the fundamental solution of \eqref{equ-DSFS} decays like $|t|^{-\frac{1}{2}}$ as $t \rightarrow \infty$. As for general $h \in (0,1)$, one can certainly obtain similar decay estimates and consequent Strichartz estimates by a scaling argument to the corresponding oscillatory integral. By this way, however, the implicit constants in such estimates always depend on $h$, which blow up as $h \rightarrow 0$. Therefore, these estimates do not apply directly when we concern the continuum limit, i.e. the convergence problem for $h\rightarrow 0$.

Uniform Strichartz estimates in $h$ are established for the first time by Y. Hong and C. Yang in \cite{HY19-DCDS}, where they considered the discrete Schr\"{o}dinger equation on $h\Z^d$ and observed that the $h$-dependence can be removed by compensating certain fractional derivatives on the right side, which brings Sobolev spaces naturally. Their strategy also applied to other discrete dispersive equations, see \cite{CQ24,CA24}.

To the best of our knowledge, there are no well-known results on uniform space-time estimates for the classical model \eqref{equ-DSFS}. In this article we aim to establish frequency localized dispersive estimates for the fundamental solution of the DSFS, serving as the first step to Strichartz estimates. Indeed, the fundamental solution of \eqref{equ-DSFS}, obtained by discrete Fourier transform, is given by
\begin{equation}\label{equ-fundamental solution}
    G(x,t) = \frac{1}{(2\pi)^d}\int_{\mathbb{T}^d_h}  e^{ix \cdot \xi+i \omega^4 t} \,d\xi,\,\,\,\,\mbox{with}\,\, \omega(h,\xi) := \frac{1}{h}\sqrt{\sum_{j=1}^{d} (2-2\cos h\xi_j)}.
\end{equation}
Here $x\in h\Z^d, t\in\R$ and $\mathbb{T}^d_h=\frac{1}{h}[-\pi,\pi]^d$. Moreover, we can decompose the frequency domain $\mathbb{T}^d_h$ by the Littlewood-Paley projections, see Section \ref{ssec-PL}, and obtain
\begin{equation}\label{equ-K}
        K_{N,h}(x,t)= \int_{\mathbb{T}^d_h} \exp \left\{ix\cdot \xi +i\frac{16}{h^4}\left(\sum_{j=1}^{d}\sin^2\left(\frac{h\xi_j}{2}\right)\right)^2\right\} \eta\left(\frac{h\xi}{N}\right)d\xi
\end{equation}
where $\eta$ is a fixed cutoff function and $N$ belongs to the set $2^{\mathbb{Z}}=\left\{2^j:j \in \Z \right\}$. $K_{N,h}$ can be regarded as an oscillatory integral with smooth phase and multiple parameters. Our first result is to determine the leading term of its asymptotic expansion when $d=2$.
\begin{theorem}\label{thm-main}
    Let $d=2$. There exists a universal constant $C$, independent of $N,h,x$ and $t$, such that
    \begin{equation*}
        |K_{N,h}(x,t)| \leq C |t|^{-\frac{1}{2}}
    \end{equation*} 
    holds for large enough $t$, $2^{\Z} \ni N \leq 1$, $h>0$ and $x\in h\Z^2$.
\end{theorem}

\begin{remark}
    The following are difficulties for the proof of the result.
    \begin{enumerate}
        \item The phase function associated with $K_{N,h}$ may have degenerate critical points. We observe that this phenomenon only happen when $N$ is away from 0, hence there are finitely many such dyadic numbers $N$. Then, we can overcome this difficulty by rescaling the phase to the standard one, i.e. $h=1$. As mentioned before, the dispersive estimate for $G$ when $h=1$ has been proved in \cite[Theorem 1.1]{C24}, where a careful analysis of degenerate critical points is needed. 
        \item For small $N$, although there are no degenerate critical points, we need to figure out the precise decay for infinitely many $N$. Inspired by \cite{CA24}, we take polar coordinates and apply the method of stationary phase to the inner integral, producing a one-variable integral consequently. We show that the second derivative of its phase never vanishes and then use van der Corput lemma. See Section \ref{ssec-complicated} for details. 
    \end{enumerate}
\end{remark}

As a corollary, we establish uniform Strichartz estimates.
\begin{theorem}\label{thm-strichartz}
    Let $h\in (0,1]$ and $d=2$. For any admissible pair $(q,r)$, i.e. $q,r\geq 2$ and $\frac{1}{q}=\frac{1}{2}\left(\frac{1}{2}-\frac{1}{r}\right)$, there exists $C=C_{q,r}>0$ independent of $h$, such that
    \begin{equation*}
        ||e^{it(-\Delta_h)^2}f||_{L^q_t L^r(h\Z^2)} \leq C||f||_{L^2(h\Z^2)}.
    \end{equation*}
\end{theorem}

\begin{remark}
    ~
    \begin{enumerate}
        \item For fixed $h$, Strichartz estimates on the lattice $h\Z^2$ are established by the first author \cite[Theorem 1.3]{C24}. We emphasize that the independece of $h$ for the implicit constants in the above theorem, where the Paley-Littlewood inequality (c.f. Lemma \ref{lem-PL}) plays an important role, is essential for the applications. 
        \item When $d=1$, similar estimates can be obtained easily by van der Corput lemma. While for $d \geq 3$, it seems very hard to analyze oscillatory integrals with degenerate phase. Therefore, estimates in higher dimensions remain open.
    \end{enumerate}
\end{remark}
On the other hand, \eqref{equ-DSFS} can be regarded as a discretization of the following semilinear fourth-order Schr\"{o}dinger equation
\begin{equation}\label{equ-SFS}
     i \partial_t u + {\Delta}^2 u=\lambda |u|^{p-1}u
\end{equation}
on the Euclidean space $\R^d$. As a typical model in physics, this equation has been introduced by V. I. Karpman \cite{K96} and A. G. Shagalov \cite{KS00} to take into account the role of small fourth-order dispersion terms in the propagation of intense laser beams in a bulk medium with Kerr nonlinearity. Since then, there are many mathematical results for \eqref{equ-SFS}, including the Strichartz estimates, global well-posedness and scattering in different dimensions. See for instance B. Pausader \cite{P07,P09}, C. Miao et al. \cite{MXZ09,MXZ11}, C. Hao et al. \cite{HHW06}, B. Guo and B. Wang \cite{GW02}.

It is natural to consider the convergence results from the solution of the discrete equation \eqref{equ-DSFS} to that of the continuous equation \eqref{equ-SFS}, which will be useful for applications such as numerical analysis. Such convergence, usually referred as the continuum limit, has been investigated for other dispersive equations since the pioneering work of K. Kirkpatrick et al. \cite{KLS13}. Later in \cite{HY19}, Y. Hong and C. Yang proved the $L^2$ strong convergence for the Schr\"{o}dinger equation after establishing uniform Strichartz estimates. Their result was then extended to the fractional Schr\"{o}dinger equation in \cite{CA24}. Recently in \cite{CQ24}, Q. Chauleur proposed a different method to deal with Schr\"{o}dinger equation, by using the growth of Sobolev norms and the Shannon interpolation instead of the linear interpolation. His method works also for the Klein-Gordon equation \cite{CQ24-arXiv}.

Our second result is to prove the continuum limit of the DSFS and determine a precise rate of convergence. We will see that the corresponding uniform Strichartz estimate (Theorem \ref{thm-strichartz}) is essential when we compare the difference between the continuous solutions and discrete ones.

To begin with, we need the discretization operator and the linear interpolation operator $p_h$, see Section \ref{sec-continuum limit} for their definitions and properties, in order to transform functions defined on the lattice and on the Euclidean space into the same setting. With them, we prove the following result:

\begin{theorem}\label{thm-continuum limit}
    Assume that
    \begin{equation}\label{equ-p}
        \left\{
        \begin{aligned}
            & 1<p<\infty ~~& \mbox{when}~~\lambda>0, \\
            & 1<p<5 ~~& \mbox{when}~~\lambda<0.
        \end{aligned}
        \right.
    \end{equation}
    Given an initial data $u_0 \in H^2(\R^2)$, and let $u(t) \in C(\R,H^2(\R^2))$ be the global solution to \eqref{equ-SFS}. Moreover, let $u_{0,h}$ be the discretization of $u_0$ as defined in \eqref{equ-inter ope}, and denote by $u_h$ the global solution to \eqref{equ-DSFS} with the initial data $u_{0,h}$. Then there exist $A,B>0$ independent of $h$ such that 
    \begin{equation*}
        ||p_hu_h(t)-u(t)||_{L^2(\R^2)} 
        \leq A h^{\frac{2}{3}} e^{B|t|} (1+||u_0||_{H^2(\R^2)})^p.
    \end{equation*}
\end{theorem}

\begin{remark}
    ~
    \begin{enumerate}
        \item The global well-posedness of \eqref{equ-DSFS} and \eqref{equ-SFS} should be addressed before proving the convergence. In fact, global solutions to both equations exist under the condition of the theorem, see Section \ref{sec-NE}.
        \item The range of the exponent $p$ comes from certain estimate for nonlinear equations, see the detailed proof in Section \ref{sec-NE}.
    \end{enumerate}
\end{remark}

This paper is orgnized as follows. In Section \ref{sec-backg} we recall basic notions such as function spaces on lattices and Paley-Littlewood operators. In Section \ref{sec-mainproof} we prove Theorem \ref{thm-main} and \ref{thm-strichartz}. In Section \ref{sec-NE} we prove some results on nonlinear equations, serving as the preparation for the continuum limit. Finally in Section \ref{sec-continuum limit}, we prove Theorem \ref{thm-continuum limit}.

\section{Preliminaries}\label{sec-backg}
\subsection{Notation}

We denote by $L^p(\R^d)$ ($L^p(h\Z^d)$ resp.) the Lebesgue space on the Euclidean space $\R^d$ (the lattice $h\Z^d$ resp.). We write the (homogeneous) Sobolev spaces on $\R^d$ and $h\Z^d$ in the similar way. $\mathcal{F}$ and $\mathcal{F}^{-1}$ are the Fourier transform on lattices and its inverse. The discrete Laplacian $\Delta_h$ has been defined in the introduction, while $\Delta$ is the standard operator on $\R^d$.

The positive constants $c$ and $C$ may vary from one line to another, but they do not depend on the stepsize $h$. During the proof, we usually omit them by using the inequality symbol $\lesssim$.  The Japanese bracket is given by $\langle t \rangle = (1+|t|^2)^{\frac{1}{2}}$.

\subsection{Basics on the discrete setting}
In this part we recall Lebesgue spaces on the lattice $h\Z^d$. The $L^p(h\Z^d)$ norm is defined by
\begin{equation*}
    ||f||_{L^p(h\Z^d)}=
    \left\{
        \begin{aligned}
            & \left\{h^d \sum_{x\in h\Z^d}|f(x)|^p\right\}^{1/p}, ~~~~~\mbox{if}~~~~~1\leq p < \infty;\\
            & \sup_{x\in h\Z^d}|f(x)|, ~~~~~\mbox{if}~~~~~p=\infty,
        \end{aligned}
    \right.
\end{equation*}
and consequently $L^p(h\Z^d) = \{f:h\Z^d \rightarrow \C, ||f||_{L^p(h\Z^d)} < \infty\}$. 

The convolution is given by
\begin{equation*}
    (f*g)(x) :=h^d \sum_{y\in h\Z^d} f(x-y)g(y).
\end{equation*}

For a rapidly decreasing function $f$, that is, $|x|^k|f(x)|\leq C_k$ for any $k\in \N$, we define its Fourier transform by 
\begin{equation}
    \mathcal{F}(f)(\xi)=h^d \sum_{x\in h\Z^d} f(x)e^{-ix\cdot\xi},~\xi \in \mathbb{T}^d_h.
\end{equation}
The inverse transform of a smooth function $f:\mathbb{T}^d_h \rightarrow \C$ is given by 
\begin{equation*}
    \mathcal{F}^{-1}(f)(x) = \frac{1}{(2\pi)^d}\int_{{T}^d_h}f(\xi)e^{ix\cdot\xi}d\xi,~x \in h\Z^d.
\end{equation*}
Note that such transforms could be extended to more general functions. Similar to the continuous case, the basic properties of $L^p(h\Z^d)$ functions and the Fourier transform still hold such as the Plancherel theorem.

If $f$ is the Cauchy data of \eqref{equ-DSFS}, the evolution operator of linear flow $e^{it(-\Delta_h)^2}$ is defined as a Fourier multiplier (see \eqref{equ-fundamental solution} for the definition of  $G$)
\begin{equation*}
    \mathcal{F}(e^{it(-\Delta_h)^2}f)(\xi) = 
    \mathcal{F}(G)\mathcal{F}(f)(\xi)= e^{it\omega(\xi)^4}\mathcal{F}(f)(\xi).
\end{equation*}


\subsection{Paley-Littlewood theory}\label{ssec-PL}

Now we recall the Paley-Littlewood decomposition and several important properties, whose proof can be found in \cite{HY19-DCDS}. We choose and fix a smooth and axisymmetric function $\varphi: \R^d \rightarrow [0,1]$ such that $\varphi(\xi) = 1$ on the square $[-1,1]^d$ but $\varphi(\xi) = 0$ on $\R^d \backslash [-2,2]^d$. For a dyadic number $N \in 2^{\Z}$, we denote
\begin{equation*}
    \psi_N(\xi) =\eta\left(\frac{h\xi}{N}\right) =\varphi \left(\frac{h\xi}{2\pi N}\right) - \varphi \left(\frac{h\xi}{\pi N}\right)
\end{equation*}
Then, we have
\begin{equation*}
    \mbox{supp} \psi_N \subset \left[-\frac{4\pi N}{h},\frac{4\pi N}{h}\right]^d \backslash \left[-\frac{\pi N}{h},\frac{\pi N}{h}\right]^d,\mbox{supp} ~\eta \subset \left[-4\pi,4\pi\right]^d \backslash \left[-\pi,\pi\right]^d,
\end{equation*}
and 
\begin{equation*}
    \sum_{2^{\Z} \ni N\leq 1} \psi_N =1.
\end{equation*}

The Paley-Littlewood operator $P_N$ is defined by a Fourier multipler,
\begin{equation*}
    \mathcal{F}(P_Nf)(\xi) = \psi_N(\xi)\mathcal{F}(f)(\xi).
\end{equation*}


The core result is the following inequality.
\begin{lemma}[Paley-Littlewood inequality]\label{lem-PL}
    Let $1<p<\infty$. For $f\in L^p(h\Z^d)$, there exist positive constants $c_p$ and $C_p$, independent of $h$, such that
    \begin{equation*}
        c_p||f||_{L^p(h\Z^d)} \leq \Big|\Big|\left(\sum_N|P_Nf|^2\right)^{\frac{1}{2}}\Big|\Big|_{L^p(h\Z^d)} \leq C_p||f||_{L^p(h\Z^d)}.
    \end{equation*}
\end{lemma}

\subsection{Sobolev spaces}\label{ssec-sobo}

On the lattice $h\Z^d$, Sobolev spaces $H^s(h\Z^d)$ can be defined by the following equivalent ways. For this purpose, we recall that differential operators $|\nabla|^s$, $\langle \nabla\rangle ^s$ and $(-\Delta_h)^s$ can be defined as Fourier multipliers:
\begin{equation*}
    \begin{aligned}
        &\mathcal{F}(|\nabla|^sf)(\xi) = |\xi|^s \mathcal{F}(f)(\xi);\\
        &\mathcal{F}(\langle \nabla\rangle ^sf)(\xi) = (1+|\xi|^2)^{\frac{s}{2}} \mathcal{F}(f)(\xi);\\
        &\mathcal{F}((-\Delta_h)^sf)(\xi)=\left\{\sum_{j=1}^{d}\frac{4}{h^2}\sin^2\left(\frac{h\xi_j}{2}\right)\right\}^{\frac{s}{2}}\mathcal{F}(f)(\xi).
    \end{aligned}
\end{equation*}
Based on this, we define $H^s(h\Z^d)$ ($\dot{H}^s(h\Z^d)$ resp.) as the Banach space equipped with the norm 
\begin{equation*}
    ||f||_{H^s(h\Z^d)}:=||\langle \nabla \rangle^s f||_{L^2(h\Z^d)}~~~(||f||_{\dot{H}^s(h\Z^d)}:=|||\nabla |^s f||_{L^2(h\Z^d)} ~~\mbox{resp.}).
\end{equation*}
In fact, we have the following norm equivalence (c.f. \cite[Proposition 6]{HY19-DCDS})
\begin{equation*}
    ||f||_{H^s(h\Z^d)} \sim ||(1-\Delta_h)^{\frac{s}{2}}f||_{L^2(h\Z^d)},  ~~~\forall s \in \R,
\end{equation*}
and
\begin{equation*}
    ||f||_{\dot{H}^1(h\Z^d)} \sim \sum_{j=1}^{d}\Big|\Big|
    \frac{f(\cdot+he_j)-f(\cdot)}{h}\Big|\Big|_{L^2(h\Z^d)}.
\end{equation*}
Similar to $H^s(\R^d)$, a lower Sobolev norm is naturally bounded above by a higher one. 

Another useful result is the Gagliardo-Nirenberg-Sobolev inequality, see \cite[Proposition 5 and 7]{HY19-DCDS}. Here we only recall the special case $d=2$.
\begin{lemma}\label{lem-GNS}
    If $\frac{1}{q}=\frac{1}{2}-\frac{\theta s}{2}$ with $\theta \in (0,1)$, then
    \begin{equation*}
        ||f||_{L^q(h\Z^2)} \lesssim ||f||_{L^2(h\Z^2)}^{1-\theta}||f||^{\theta}_{\dot{H}^s(h\Z^2)}.
    \end{equation*}
    If $\frac{1}{q}=\frac{1-s}{2}$ and $q < \infty$, then
    \begin{equation*}
        ||f||_{L^q(h\Z^2)} \lesssim ||f||_{H^s(h\Z^2)}.
    \end{equation*}
    Moreover, if $q=\infty$ and $s>1$, we have
    \begin{equation*}
        ||f||_{L^{\infty}(h\Z^2)} \lesssim ||f||_{H^s(h\Z^2)}.
    \end{equation*}
\end{lemma}

\section{Proof of Theorem \ref{thm-main} and \ref{thm-strichartz}}\label{sec-mainproof}

Recall that $K_{N,h}$ is defined in \eqref{equ-K}. For each $x \in h\Z^d$, we write $x=hy$ where $y \in \Z^d$. Then
\begin{equation*}
    \begin{aligned}
        K_{N,h}(x)
        & =K_{N,h}(hy)\\
        & =\int_{\mathbb{T}^d_h} \exp \left\{ihy\cdot \xi +i\frac{t}{h^4} \left(\sum_{j=1}^{d}2-2\cos(h\xi_j)\right)^2\right\} \eta\left(\frac{h\xi}{N}\right)d\xi\\
        &=h^{-d} \int_{\mathbb{T}^d} \exp \left\{iy\cdot \xi +i\frac{t}{h^4} \left(\sum_{j=1}^{d}2-2\cos(\xi_j)\right)^2\right\} \eta\left(\frac{\xi}{N}\right)d\xi \\
        &=h^{-d} \left(\int_{\mathbb{T}^d} \exp \left\{iy\cdot \xi +i\frac{t}{h^4}\left(\sum_{j=1}^{d}2-2\cos(\xi_j)\right)^2\right\} d\xi\right) *_{y} \left(\int_{\mathbb{T}^d} e^{iy\cdot \xi} ~\eta\left(\frac{\xi}{N}\right)d\xi\right)\\
        &=h^{-d}~ G \left(y,\frac{t}{h^4}\right) *_y \left(\int_{\mathbb{T}^d} e^{iy\cdot \xi} ~\eta\left(\frac{\xi}{N}\right)d\xi\right).
    \end{aligned}
\end{equation*}
By Young's inequality, 
\begin{equation*}
    ||K_{N,h}(\cdot,t)||_{L^{\infty}(h\Z^d)} \leq h^{-d} \Big|\Big|G \left(\cdot,\frac{t}{h^4}\right)\Big|\Big|_{L^{\infty}(\Z^d)} \Big|\Big|\mathcal{F}^{-1}\left(\eta\left(\frac{\cdot}{N}\right)\right)\Big|\Big|_{L^1(\Z^d)}.
\end{equation*}

Now we divide the proof into cases for large or small $N \in 2^{\Z}$.
\subsection{Large N}
By the dispersive estimate \cite[Theorem 1.1]{C24}, we know that 
\begin{equation*}
    \Big|\Big|G \left(\cdot,\frac{t}{h^4}\right)\Big|\Big|_{L^{\infty}(h\Z^2)} \leq C \left(\frac{|t|}{h^4}\right)^{-\frac{1}{2}}.
\end{equation*}
And hence when $d=2$,
\begin{equation*}
    ||K_{N,h}(\cdot,t)||_{L^{\infty}(h\Z^d)} \lesssim h^{-2} \left(\frac{|t|}{h^4}\right)^{-\frac{1}{2}} \Big|\Big|\mathcal{F}^{-1}\left(\eta\left(\frac{\cdot}{N}\right)\right)\Big|\Big|_{L^1(\Z^2)} \leq C_N |t|^{-\frac{1}{2}}.
\end{equation*}
On the other hand, if we assume for now that $2^{\Z} \ni N \in (N_0,1)$ with $N_0$ to be determined later, we only need to deal with a finite number of $N$. Consequently, by adjusting the constant $C$ in Theorem \ref{thm-main} properly, we obtain the result.

\subsection{Small N}
In this part, we keep in mind that $N \leq N_0$ and $d=2$.

Letting $\xi \rightarrow \frac{N\xi}{h}$, we have 
\begin{equation*}
    K_{N,h}(x,t)  = \left(\frac{N}{h}\right)^2 \int_{[-\frac{\pi}{N},\frac{\pi}{N}]^2} \exp \left\{i\frac{xN}{h}\cdot \xi +i\frac{16t}{h^4}\left(\sin^2\left(\frac{N\xi_1}{2}\right)+\sin^2\left(\frac{N\xi_2}{2}\right)\right)^2\right\} \eta(\xi)d\xi.
\end{equation*} 
The support of $\eta$ is contained in $\{ \frac{\pi}{2} \leq |\xi_i| \leq 2\pi \}$. Therefore, we turn to consider a related integral, for $x \in \R^2$,
\begin{equation}\label{equ-I_N}
    I_{N}(x,t) = \int_{\R^2} \exp \left\{ix\cdot \xi +it\left(\sin^2\left(\frac{N\xi_1}{2}\right)+\sin^2\left(\frac{N\xi_2}{2}\right)\right)^2\right\} \eta(\xi)d\xi.
\end{equation}
Note that for small $N$ such that $\frac{\pi}{N} \geq 2\pi$,
\begin{equation*}
    K_{N,h}(x,t)=\left(\frac{N}{h}\right)^2 I_N\left(\frac{xN}{h},\frac{16t}{h^4}\right).
\end{equation*}
In order to prove Theorem \ref{thm-main}, it suffices to prove the following.
\begin{theorem}\label{thm-main 2}
    There exists a constant $C$, which is independent of $N$, $x$ and $t$, such that
    \begin{equation*}
        |I_N(x,t)| \leq C N^{-4}|t|^{-1}
    \end{equation*} 
    holds for large enough $t$, $2^{\Z} \ni N \leq N_0$, and $x\in \R^2$.
\end{theorem}

\begin{proof}[Proof of Theorem \ref{thm-main} assuming Theorem \ref{thm-main 2}]
    By the relation between $K_{N,h}$ and $I_N$, we know 
    \begin{equation*}
        ||K_{N,h}(\cdot,t)||_{L^{\infty}(h\Z^2)} \leq C\left(\frac{N}{h}\right)^2 \left(\frac{|t|}{h^4}\right)^{-1}N^{-4}=C N^{-2}h^2|t|^{-1}.
    \end{equation*}
    On the other hand, we have the trivial estimate
    \begin{equation*}
        ||K_{N,h}(\cdot,t)||_{L^{\infty}_{h}} \leq C\left(\frac{N}{h}\right)^2.
    \end{equation*}
    Interpolating with these two inequalities, we finish the proof.
\end{proof}

\subsection{Proof of Theorem \ref{thm-main 2}}\label{ssec-complicated}
Beginning with \eqref{equ-I_N}, we let $z_i=\frac{2}{N}\sin\left(\frac{N\xi_i}{2}\right)$, and hence $\xi_i=\frac{2}{N}\arcsin\left(\frac{Nz_i}{2}\right)$. If $N$ is sufficiently small, this map is indeed injective. The Jacobian is 
\begin{equation*}
    J(z) = \left(\sqrt{1-\left(\frac{Nz_1}{2}\right)}\sqrt{1-\left(\frac{Nz_2}{2}\right)}\right)^{-1}.
\end{equation*}
By the support of $\eta$, $J$ is a smooth function. Then we have
\begin{equation*}
    \begin{aligned}
        I_N(x,t)
        &=\int_{\R^2} \exp \left\{ix\cdot \xi(z) +it\left( \left(\frac{Nz_1}{2}\right)^2+\left(\frac{Nz_2}{2}\right)^2\right)^2\right\} \eta(\xi(z))J(z)dz\\
        &=\int_{\R^2} \exp \left\{ix\cdot \xi(z) +itN^4|z|^4 \right\}\eta(\xi(z))J(z)dz.
    \end{aligned}
\end{equation*}

We then take polar coordinates, that is, $z=\rho e^{i\phi}$, and write $x=r e^{i\theta}$ with $0\leq r <\infty$ and $\theta \in [0,2\pi]$. Now we pay attention to the range of new variables. By $\frac{N\pi}{4} \leq \Big|\arcsin\left(\frac{Nz_1}{2}\right)\Big|\leq N\pi$ and an elementary inequality $a \leq \arcsin a \leq 2a$ (when $a \in [0,\frac{\sqrt{2}}{2}]$), we know $\frac{\pi}{4}\leq \rho \leq 2\pi$.

Hence, we write
\begin{equation*}
    \begin{aligned}
        I_N(x,t)
        &= \int_{\frac{\pi}{4}}^{2\pi} \int_{0} ^{2\pi} e^{ ir\rho \Phi(\phi)+itN^4\rho^4} \eta(\xi(\rho,\phi))J(\rho,\phi)\rho d\rho d\phi \\
        &=\int_{\frac{\pi}{4}}^{2\pi} e^{itN^4\rho^4} \rho d\rho \int_{0} ^{2\pi} e^{ir\rho \Phi(\phi)}\eta(\rho,\phi)J(\rho,\phi) d\phi \\
        & := \int_{\frac{\pi}{4}}^{2\pi} e^{itN^4\rho^4} \rho G(\rho,N,r,\theta) d\rho
    \end{aligned}
\end{equation*}
where the phase
\begin{equation*}
    \Phi(\phi)=\frac{2}{N\rho}\left[\cos \theta \arcsin\left(\frac{N\rho\cos\phi}{2}\right)+\sin \theta \arcsin\left(\frac{N\rho\sin\phi}{2}\right)\right]
\end{equation*}
and $G$ is a one-variable integral on $\phi$.

By integrating by parts, we get
\begin{equation*}
    \begin{aligned}
        I_N(x,t)
        &=-\frac{1}{4itN^4} \int_{\frac{\pi}{4}}^{2\pi}e^{itN^4\rho^4} \frac{d}{d\rho} (\rho^{-2} G(\rho,N,r,\theta)) d\rho\\
        &=-\frac{1}{4itN^4} \int_{\frac{\pi}{4}}^{2\pi}e^{itN^4\rho^4} \left(-2\rho^{-3}G+\rho^{-2}\frac{dG}{d\rho}\right)
         d\rho.
    \end{aligned}
\end{equation*}
Note that $\rho$ is bounded and $|G|\leq C(N_0)$ (see the formula of $J$ and recall that $N$ has an upper bound $N_0$). A harder part is 
\begin{equation*}
    \begin{aligned}
        \frac{dG}{d\rho}
        &=\int_{0} ^{2\pi} \frac{d}{d\rho} (e^{ir\rho \Phi(\phi)}) \eta J d\phi+\int_{0} ^{2\pi} e^{ir\rho \Phi(\phi)}\frac{d}{d\rho} (\eta J)\\
        &:= (E1) + (E2).
    \end{aligned}
\end{equation*}
For $(E2)$, a trivial estimate gives $(E2) \leq C(N_0)$. While for $(E1)$, the chain rule implies $\partial_{\rho}=\cos \phi \partial_{z_1}+\sin \phi \partial_{z_2}$, and a direct calculation gives $\partial_{z_i}e^{ix\cdot \xi(z)} \lesssim |x| = r$. Therefore, as long as $r$ has an upper bound, we have
\begin{equation*}
    |I_{N}(x,t)| \lesssim N^{-4} t^{-1}.
\end{equation*}

On the other hand, as $r \rightarrow \infty$, we turn to a careful analysis about $G$. Recall that 
\begin{equation*}
    G(\rho,N,r,\theta) = \int_{0} ^{2\pi} e^{ir\rho \Phi(\phi)}\eta(\rho,\phi)J(\rho,\phi) d\phi,
\end{equation*}
in which we regard $r,\rho$ as parameters. To understand this integral, we calculate
\begin{equation*}
    \Phi'(\phi) = -\frac{\cos\theta\sin\phi}{\sqrt{1-\left(\frac{N\rho\cos\phi}{2}\right)^2}}+\frac{\sin\theta\cos\phi}{\sqrt{1-\left(\frac{N\rho\sin\phi}{2}\right)^2}}.
\end{equation*}
Thus, the critical points are those $\phi$ such that
\begin{equation*}
    \tan \theta = \tan \phi \left(\frac{1-\left(\frac{N\rho\sin\phi}{2}\right)^2}{1-\left(\frac{N\rho\cos\phi}{2}\right)^2}\right)^{1/2} := \tan \phi ~g(\rho,\phi).
\end{equation*}
About $g(\rho,\cdot)$, we know  it is positive, even and $\pi$-periodic. Roughly speaking, its graph (between $-\frac{\pi}{2}$ and $\frac{\pi}{2}$) looks like a cosine function with a strictly positive boundary point. Moreover, noticing that $\tan \phi<g(\rho,\cdot)\tan \phi <1$ on $(0,\frac{\pi}{4})$, $g(\rho,\cdot)\tan \phi <\tan \phi$ on $(\frac{\pi}{4},\frac{\pi}{2})$, and $g(\rho,\cdot)\tan \phi$ is an odd $\pi$-periodic function, we obtain the following conclusion. If $\tan \theta \in [0,\infty]$, there always exist two critical points, denoted by $\phi_+ \in [0,\frac{\pi}{2}]$ and $\phi_{-} \in [\pi,\frac{3\pi}{2}]$ such that 
\begin{equation*}
    \Phi'(\phi_+)=\Phi'(\phi_-)=0~~~ \mbox{and}~~~ \phi_++\pi=\phi_-.
\end{equation*}
The case $\tan \theta \in [-\infty,0]$ is similar by symmetry. We remark here that for certain $\theta_{\pm}$ such that $\tan \theta_+ = \tan \theta_-$, we know that $0 \leq |\theta_{\pm}-\phi_{\pm}| \leq \frac{\pi}{4}$.

From this point, we deduce that the critical points here are non-degenerate. To see this, we calculate that
\begin{equation*}
    \Phi''(\phi) = -\left(1-\left(\frac{N\rho}{2}\right)^2\right)\left[\frac{\cos\theta\cos\phi}{\left(1-\left(\frac{N\rho\cos\phi}{2}\right)^2\right)^{\frac{3}{2}}}+\frac{\sin\theta\sin\phi}{\left(1-\left(\frac{N\rho\sin\phi}{2}\right)^2\right)^{\frac{3}{2}}}\right].
\end{equation*}
Since $\rho$ is bounded and $N$ is close enough to 0, we know 
\begin{equation*}
    |\Phi''(\phi_{\pm})| \approx |\cos (\theta-\phi_{\pm})| \gtrsim \cos \left(\frac{\pi}{4}\right) \gtrsim 1.
\end{equation*} 
Inspired by this, we take two cut-off functions, $\chi_{\pm}$. Let $\chi_{+}=1$ on the first quadrant and be supported in $(-\frac{\pi}{8},\frac{5\pi}{8})$. Let $\chi_{-}=1$ on the third quadrant and be supported in $(\frac{7\pi}{8},\frac{13\pi}{8})$. These two support sets are separated from each other and we define $\chi_0=1-\chi_{+}-\chi_{-}$. One may verify that the derivatives in $\phi$ of $\chi_{\pm}\eta J$ and $\chi_{0}\eta J$ are bounded, which is independent of $N,\rho$.

Now we write $G$ as the sum of the following three integrals.
\begin{equation*}
    \begin{aligned}
        G(t,r,\rho) = 
        &=\int_{0}^{2\pi} e^{ir\rho\Phi(\phi)}\eta J \chi_{+} d\phi + \int_{0}^{2\pi} e^{ir\rho\Phi(\phi)}\eta J \chi_{-} d\phi +\int_{0}^{2\pi} e^{ir\rho\Phi(\phi)}\eta J \chi_{0} d\phi\\
        &:=G_{+} + G_{-} + G_{0}
    \end{aligned}
\end{equation*}
We write $\lambda=r\rho$ as the parameter in $G$. Away from the critical points, we expect a rapid decay in $r\rho$ for $G_0$, so thet we claim (note that $\rho$ can be regarded as a fixed constant)
\begin{equation*}
    |\partial_{\lambda}G_0| \lesssim \lambda^{-1}.
\end{equation*}
This can be obtained from
\begin{equation*}
    \partial_{\lambda}G_0 = i \int_{0}^{2\pi}e^{i\lambda\Phi(\phi)}\Phi(\phi)\eta J \chi_{0} d\phi =
    -\frac{1}{\lambda}\int_{0}^{2\pi}e^{i\lambda\Phi(\phi)} \frac{d}{d\phi}\left(\frac{\Phi\eta J \chi_{0}}{\Phi'}\right)d\phi.
\end{equation*}
Since $\Phi$, $\Phi'$ and $\Phi''$ are bounded and nonvanishing, we prove the claim. Consequently, we know that the integral $I_N$ with $G$ replaced by $G_0$ admits an estimate consistently via integrating by parts, since
\begin{equation*}
    |\partial_{\rho}G_0| = r|\partial_{\lambda}G_0|\lesssim \frac{r}{\lambda} = \frac{1}{\rho}.
\end{equation*}

Now we deal with $G_+$, and $G_-$ can be treated in the same way. By the stationary phase formula (\cite[Chapter 8]{S93}), 
we write, as $\lambda \rightarrow \infty$,
\begin{equation*}
    G_+(\lambda) = \sqrt{\frac{2\pi}{|\partial^2_{\phi}\Phi(\phi_+)|}}e^{i[\lambda\Phi(\phi_+)-\frac{\pi}{4}]}\eta(\phi_+)J(\phi_+)\chi_+(\phi_+)\lambda^{-\frac{1}{2}}+\widetilde{G_+}(\lambda).
\end{equation*}
The error term $\widetilde{G_+}(\lambda)$ satisfies
\begin{equation*}
    \Big| \frac{d}{d\lambda} \widetilde{G_+}(\lambda) \Big| \lesssim \lambda^{-\frac{3}{2}}
\end{equation*}
as long as $\lambda$, and hence $r$ (note that $\rho$ has an upper bound), is large enough. As a result,
\begin{equation*}
    |\partial_{\rho}\widetilde{G_+}| \lesssim \frac{r}{\lambda^{\frac{3}{2}}} = \frac{1}{\rho^{\frac{3}{2}}r^{\frac{1}{2}}}.
\end{equation*}
By this, the integral $I_N$ with $G$ replaced by $\widetilde{G_+}$ admits again a similar estimate.

It remains to consider the case for which we insert the leading term of the asymptotic expansion of $G_+$ into $I_N$, given by $(\lambda = r \rho)$
\begin{equation*}
    \begin{aligned}
        \widetilde{I_N}    
        &= \int_{\frac{\pi}{4}}^{2\pi}e^{itN^4\rho^4+i\lambda\Phi(\phi_+)}\eta(\phi_+)J(\phi_+)\chi_+(\phi_+)\lambda^{-\frac{1}{2}}\rho|\partial^2_{\phi}\Phi(\phi_+)|^{-\frac{1}{2}} d\rho\\
        &=r^{-\frac{1}{2}}\int_{\frac{\pi}{4}}^{2\pi}e^{itN^4\rho^4+ir\rho\Phi(\phi_+)} \rho^{\frac{1}{2}}|\partial^2_{\phi}\Phi(\phi_+)|^{-\frac{1}{2}}\eta(\phi_+)J(\phi_+)\chi_+(\phi_+) d\rho.
    \end{aligned}
\end{equation*}
Note that the critical point $\phi_{+}$ depends on $\rho$. We then write $r=vt$, where $v \in \R^2$, so that
\begin{equation*}
    \begin{aligned}
        \widetilde{I_N}
        &=r^{-\frac{1}{2}}\int_{\frac{\pi}{4}}^{2\pi} e^{it(N^4\rho^4+v\rho\Phi(\phi_+(\rho)))}\rho^{\frac{1}{2}}|\partial^2_{\phi}\Phi(\phi_+)|^{-\frac{1}{2}}\eta(\phi_+)J(\phi_+)\chi_+(\phi_+) d\rho\\
        &:=r^{-\frac{1}{2}}\int_{\frac{\pi}{4}}^{2\pi}
        e^{itS(\rho)}a(\rho)d\rho.
    \end{aligned}
\end{equation*}
This is a one-variable oscillatory integral, which can be estimated by the van der Corput lemma. To apply it, we should calculate the derivatives of 
\begin{equation*}
    S(\rho)=N^4\rho^4+\frac{2v}{N}\left[\cos \theta \arcsin\left(\frac{N\rho\cos\phi_+}{2}\right)+\sin \theta \arcsin\left(\frac{N\rho\sin\phi_+}{2}\right)\right].
\end{equation*}
Then, by $\Phi'(\phi_+)=0$,
\begin{equation*}
    S'(\rho) = 4N^4\rho^3+v\left[\frac{\cos\theta\cos\phi_+}{\left(1-\left(\frac{N\rho\cos\phi_+}{2}\right)^2\right)^{\frac{1}{2}}}+\frac{\sin\theta\sin\phi_+}{\left(1-\left(\frac{N\rho\sin\phi_+}{2}\right)^2\right)^{\frac{1}{2}}}\right].
\end{equation*}
Moreover, we first calculate by differentiating $\Phi'(\phi_+)=0$ that
\begin{equation}\label{equ-phi-rho}
    \frac{\partial\phi_+(\rho)}{\partial{\rho}}
    =
    \frac{4N^2\rho\sin(4\phi_+)}{(4-N^2\rho^2)(N^2\rho^2(1-\cos(4\phi_+))-16)}.
\end{equation}
With this formula, we have
\begin{equation}\label{equ-S''}
    \begin{aligned}
        S''(\rho)
        =&12N^4\rho^2+v\cdot \frac{N^2\rho}{4}\left[\frac{\cos\theta\cos^3\phi_+}{\left(1-\left(\frac{N\rho\cos\phi_+}{2}\right)^2\right)^{\frac{3}{2}}}+\frac{\sin\theta\sin^3\phi_+}{\left(1-\left(\frac{N\rho\sin\phi_+}{2}\right)^2\right)^{\frac{3}{2}}}\right] \\
        &+ v \left[\frac{-\cos\theta\sin\phi_+}{\left(1-\left(\frac{N\rho\cos\phi_+}{2}\right)^2\right)^{\frac{3}{2}}}+\frac{\sin\theta\cos\phi_+}{\left(1-\left(\frac{N\rho\sin\phi_+}{2}\right)^2\right)^{\frac{3}{2}}}\right]\frac{\partial\phi_+(\rho)}{\partial{\rho}}.
    \end{aligned}
\end{equation}
Next, we will use a tedious calculation to show that 
\begin{equation*}
    |S''(\rho)| \gtrsim N^4.    
\end{equation*}
By \eqref{equ-S''}, we easily get
\begin{equation*}
    \begin{aligned}
        &\left(S''(\rho)-12N^4\rho^2\right)\frac{1}{v}\\
        =&~\frac{N^2\rho}{4}\left[\frac{\cos\theta\cos^3\phi_+}{\left(1-\left(\frac{N\rho\cos\phi_+}{2}\right)^2\right)^{\frac{3}{2}}}+\frac{\sin\theta\sin^3\phi_+}{\left(1-\left(\frac{N\rho\sin\phi_+}{2}\right)^2\right)^{\frac{3}{2}}}\right]\\
        &+\left[\frac{-\cos\theta\sin\phi_+}{\left(1-\left(\frac{N\rho\cos\phi_+}{2}\right)^2\right)^{\frac{3}{2}}}+\frac{\sin\theta\cos\phi_+}{\left(1-\left(\frac{N\rho\sin\phi_+}{2}\right)^2\right)^{\frac{3}{2}}}\right]\frac{\partial\phi_+(\rho)}{\partial{\rho}}\\
        :=&~\mathbf{I}\times \frac{N^2\rho}{4}+ \mathbf{II}\times \frac{\partial\phi_+(\rho)}{\partial{\rho}}.
    \end{aligned}
\end{equation*}
We observe from the formula of $S'(\rho)$ that there are no critical points if $v\gg N^4$ or $v\ll N^4$. Therefore, we may suppose $v \approx N^4$. Since $\mathbf{I} \approx 1$, $\mathbf{II} \approx 1$ and $\frac{\partial\phi_+(\rho)}{\partial{\rho}} \lesssim N^2$ by \eqref{equ-phi-rho}, we deduce that
\begin{equation*}
    |S''(\rho)| \geq 12N^4\rho^2 - v N^2 \gtrsim 12\left(\frac{\pi}{4}\right)^2N^4-N^6 \gtrsim N^4 .  
\end{equation*}
By the van der Corput lemma (\cite[Chapter 8]{S93}) and $r=|x|=vt$,
\begin{equation*}
    \begin{aligned}
        |\widetilde{I_N}|
        &=r^{-\frac{1}{2}}
        \Bigg|\int_{\frac{\pi}{4}}^{2\pi}e^{itS(\rho)}a(\rho)d\rho \Bigg|\\
        &\leq C r^{-\frac{1}{2}} \left(tN^4\right)^{-\frac{1}{2}}\left(||a||_{L^{\infty}([\frac{\pi}{4},2\pi])}+||\partial_{\rho}a||_{L^{1}([\frac{\pi}{4},2\pi])}\right)\\
        &\lesssim (N^4t)^{-\frac{1}{2}}t^{-\frac{1}{2}}N^{-2} \lesssim t^{-1}N^{-4}.
    \end{aligned}
\end{equation*}

Combining all above estimates, we finally complete the proof.

\subsection{Proof of Theorem \ref{thm-strichartz}}\label{ssec-Stri}
A trivial estimate gives the following energy estimate
\begin{equation*}
    ||e^{it(-\Delta_h)^2}P_Nf||_{L^2(h\Z^2)} \leq ||P_Nf||_{L^2(h\Z^2)} \lesssim ||f||_{L^2(h\Z^2)}.
\end{equation*}
Together with Theorem \ref{thm-main}, it follows from \cite[Theorem 1.2]{KT98} that for an admissible pair $(q,r)$, the following frequency localized Strichartz estimate holds
\begin{equation}\label{equ-inpf1}
    ||e^{it(-\Delta_h)^2}P_Nf||_{L_t^qL^r(h\Z^2)} \leq ||f||_{L^2(h\Z^2)}.
\end{equation}
Now, we define a cutoff function $\widetilde{\eta}$ such that $\widetilde{\eta}=1$ on the support of $\eta$ in the Paley-Littlewood theory. Also, let $\widetilde{P_N}$ be the Fourier multipler of symbol $\widetilde{\eta}(\frac{h\xi}{N})$. We point out that all the results above are valid if we replace $P_N$ with $\widetilde{P_N}$, and $P_N=\widetilde{P_N}P_N$ in the sense of operators. Therefore, we have by the Minkowski inequality and Lemma \ref{lem-PL},
\begin{equation*}
    \begin{aligned}
        ||e^{it(-\Delta_h)^2}f||_{L_t^qL^r(h\Z^2)}
        &\lesssim \Bigg|\Bigg|\left\{\sum_{N\leq 1}|e^{it(-\Delta_h)^2}P_Nf|^2\right\}^{\frac{1}{2}}\Bigg|\Bigg|_{L_t^qL^r(h\Z^2)}\\
        &\lesssim \left\{\sum_{N\leq 1}||e^{it(-\Delta_h)^2}P_Nf||_{L_t^qL^r(h\Z^2)}^2\right\}^{\frac{1}{2}} = \left\{\sum_{N\leq 1}||e^{it(-\Delta_h)^2}\widetilde{P_N}P_Nf||_{L_t^qL^r(h\Z^2)}^2\right\}^{\frac{1}{2}}\\
        &\lesssim \left\{\sum_{N\leq 1}||P_Nf||_{L^2(h\Z^2)}^2\right\}^{\frac{1}{2}} \lesssim ||f||_{L^2(h\Z^2)}.
    \end{aligned}
\end{equation*}
We have used \eqref{equ-inpf1} in the third inequality. The proof is finished.

As a corollary, we prove that the solution to discrete linear fourth-order Schr\"{o}dinger flow satisfies a time-averaged uniform bound.
\begin{corollary}\label{cor-time-aver-disc}
    We have 
    \begin{equation*}
        ||e^{it(-\Delta_h)^2}f||_{L_t^{\infty}L^{\infty}(h\Z^2)}
        \leq C_0 ||f||_{H^2(h\Z^2)}
    \end{equation*}
\end{corollary}
\begin{proof}
    To see this, we recall the Sobolev embedding (Lemma \ref{lem-GNS}) which states that $W^{s,2}(h\Z^2) \hookrightarrow L^{\infty}(h\Z^2)$ as long as $s>1$. Then we have
    \begin{equation*}
        \begin{aligned}
            ||e^{it(-\Delta_h)^2}f||_{L_t^{\infty}L^{\infty}(h\Z^2)}
        &\leq C ||e^{it(-\Delta_h)^2}f||_{L_t^{\infty}W^{2,2}(h\Z^2)} 
        =C||e^{it(-\Delta_h)^2}(1-\Delta_h)f||_{L_t^{\infty}L^2(h\Z^2)}\\
        & \lesssim ||(1-\Delta_h)f||_{L^2(h\Z^2)}=C||f||_{H^2(h\Z^2)}.
        \end{aligned}
    \end{equation*}
    In the last inequality we choose $(q,p)=(\infty,2)$ as an admissible pair and apply Theorem \ref{thm-strichartz}.
\end{proof}

\section{nonlinear equations}\label{sec-NE}

Before proving the continuum limit, it is natural to settle the well-posedness of corresponding equations, which is the aim of this section. Let's begin with the global solution and conservation laws of the DSFS.
\begin{theorem}\label{thm-global of DSDF}
    Suppose that $d \in \N$ and $p>1$. For any initial data $f \in L^2(h\Z^d)$, there exists a unique global solution $u(t) \in C(\R,L^2(h\Z^d))$ to \eqref{equ-DSFS}. Besides, it conserves the mass $M(u) := ||u||_{L^2(h\Z^d)}^2$ and the energy
    $E(u):= \frac{1}{2}||u||_{\dot{H}^2(h\Z^d)}^2+\frac{\lambda}{p+1}||u||^{p+1}_{L^{p+1}(h\Z^d)}$. 

    Moreover, suppose that $d=2$, $p$ is given in \eqref{equ-p}, and the initial data $f \in H^2(h\Z^2)$. The solution above satisfies
    \begin{equation*}
        ||u||_{L_t^{\infty} ([-T,T],L^{\infty}(h\Z^2))} \lesssim  ||f||_{H^2(h\Z^2)} ~~\mbox{for any}~~T>0. 
    \end{equation*}
\end{theorem}

\begin{proof}
    The proof of the $L^2$ global well-posedness is standard, as is shown in \cite[Proposition 4.1]{HY19}. We focus on the second assertion. We choose a sufficiently small interval $I=[-\tau,\tau]$ (with $\tau$ to be chosen later) and define
    a map
    \begin{equation}
        \Gamma(u):= e^{-it(-\Delta_h)^2}f-i\lambda\int_0^te^{-i(t-s)(-\Delta_h)^2}(|u|^{p-1}u)(s)ds
    \end{equation} 
    on a complete metric space
    \begin{equation*}
        X:=\left\{
            u:h\Z^2 \rightarrow \C: ||u||_{C_t(I;H^2(h\Z^2))}+||u||_{L_t^{\infty}(I;L^{\infty}(h\Z^d))} \leq 2(1+C_0)||f||_{H^2(h\Z^d)}
        \right\}
    \end{equation*}
    with the norm $||\cdot||_{C_t(I;L^2(h\Z^d))}$. Here $C_0$ is the constant in Corollary \ref{cor-time-aver-disc}. Consequently, we can prove that $\Gamma$ is contractive on $X$, under the condition $1<p<\infty$, as in \cite[Proposition 4.2]{HY19}.

    To extend the time interval, we need to show $||u(t)||_{H^2(h\Z^2)}$, or $||(-\Delta_h)u(t)||_{L^2(h\Z^2)}$ by the norm equivalence and the mass conservation law, is bounded globally in time. If $\lambda>0$, the energy conservation yields the result. If $\lambda<0$, we have (again by Lemma \ref{lem-GNS} with $\theta = \frac{1}{2}-\frac{1}{p+1}$)
    \begin{equation*}
        \frac{1}{2}||(-\Delta_h)u(t)||_{L^2(h\Z^2)}^2 \leq E(f)+CM(f) ||(-\Delta_h)u(t)||_{L^2(h\Z^2)}^{\frac{p-1}{2}},
    \end{equation*}
    implying $||(-\Delta_h)u(t)||_{L^2(h\Z^2)}<\infty$ as long as $\frac{p-1}{2}<2$, i.e. $p<5$.
\end{proof}

For the semilinear fourth-order Schr\"{o}dinger equation on $\R^d$, the basic theory has been established.
\begin{theorem}\label{thm-global of SDF}
    Suppose that $d \in \N$ and $p>1$. For any initial data $f \in H^2(\R^d)$, there exists a unique global solution $u(t) \in C(\R,H^2(\R^d))$ to \eqref{equ-SFS}. Besides, it conserves the mass $M(u) := ||u||_{L^2(\R^2)}^2$ and the energy
    $E(u):= \frac{1}{2}||u||_{\dot{H}^2(\R^2)}^2+\frac{\lambda}{p+1}||u||^{p+1}_{L^{p+1}(\R^2)}$. 

    Moreover, when $d=2$ and $p$ is as in \eqref{equ-p}, the solution above satisfies
    \begin{equation*}
        ||u||_{L_t^{\infty} ([-T,T],L^{\infty}(\R^2))} \lesssim ||f||_{H^2(\R^2)} ~~\mbox{for any}~~T>0. 
    \end{equation*}
\end{theorem}

\begin{proof}
    The proof is parallel to that of \cite[Appendix A]{HY19}. The global well-posedness has been established in \cite[Proposition 4.1]{P07}. The Strichartz estimate plays an important role and implies for some sufficiently small $\tau>0$
    \begin{equation*}
        ||u||_{L^{\infty}_t((-\tau,\tau),H^2(\R^2))} \lesssim ||f||_{H^2(\R^2)}.
    \end{equation*}
    Consequently, the choice of $p$ and $\lambda$ ensures that $||u(t)||_{H^2(\R^2)}$ is bounded by a constant depending only on $M(f)$ and $E(f)$, hence the interval can be extended arbitrarily to get 
    \begin{equation*}
        ||u||_{L^{\infty}_t((-T,T),H^2(\R^2))} \lesssim ||f||_{H^2(\R^2)}.
    \end{equation*}
    Based on the Gagliardo-Nirenberg inequality, see e.g, \cite[Theorem 1.1]{FFR21}, we have
    \begin{equation*}
        ||u||_{L^{\infty}(\R^2)} \lesssim ||u||_{\dot{H}^2(\R^2)}^{\frac{1}{2}} ||u||_{L^2(\R^2)}^{\frac{1}{2}} \lesssim ||u||_{H^2(\R^2)}.
    \end{equation*}
    Combining two inequalities above, we get
    \begin{equation*}
        ||u||_{L_t^{\infty} ([-T,T],L^{\infty}(\R^2))} \lesssim ||u||_{L_t^{\infty} ([-T,T],H^2(\R^2))} \lesssim
        ||f||_{H^2(\R^2)}.
    \end{equation*}
\end{proof}

\section{continuum limit}\label{sec-continuum limit}
Finally we are ready to prove the continuum limit. To begin with, we extend some properties of discretization and linear interpolation. In \cite{HY19}, Y. Hong and C. Yang established such properties mainly on $H^{\alpha}$ when $\alpha \in [0,1]$. In our case, we should focus on $H^2$.

For $f:\R^d \rightarrow \C$, it can be discretized by 
\begin{equation}\label{equ-disc ope}
    f_h(y):=\frac{1}{h^d}\int_{y+[0,h)^d}f(z)dz,~~\forall y \in h\Z^d.
\end{equation}
While for $g:h\Z^d \rightarrow \C$, a linear interpolation 
\begin{equation}\label{equ-inter ope}
    p_h(g) (z):=g(y)+(D_h^+g)(y)\cdot (z-y), ~~\forall z \in y + [0,h)^d,
\end{equation}
makes it continuous on the entire $\R^d$, where the $j$th component of the vector $(D_h^+g)(y)$ is $h^{-1}(g(y+he_j)-g(y))$.

\begin{lemma}\label{lem-bddness}
    Suppose that $f \in H^2(\R^d)$ and $g \in H^2(h\Z^d)$, then
    \begin{enumerate}
        \item $||f_h||_{\dot{H}^2(h\Z^d)} \lesssim ||f||_{\dot{H}^2(\R^d)}$, and hence $||f_h||_{H^2(h\Z^d)} \lesssim ||f||_{H^2(\R^d)}$.
        \item $||p_hg||_{\dot{H}^2(\R^d)} \lesssim ||g||_{\dot{H}^2(h\Z^d)}$, and hence $||p_hg||_{H^2(\R^d)} \lesssim ||g||_{H^2(h\Z^d)}$.
        \item $||p_hf_h-f||_{L^2(\R^d)} \lesssim h ||f||_{H^2(\R^d)}$.
    \end{enumerate}
\end{lemma}

\begin{proof}
    For (1), it follows from the same argument in \cite[Lemma 5.1]{HY19}.\\
    For (2), we observe that for any piecewise linear function $g:h\Z^d \rightarrow \C$, we always have
    \begin{equation*}
        ||p_hg||_{\dot{H}(\R^d)}=0 \lesssim ||g||_{\dot{H}^2(h\Z^d)}.
    \end{equation*}
    For (3), we combine \cite[Proposition 5.3]{HY19} when $\alpha=1$ with a trivial embedding $H^2(\R^d) \hookrightarrow H^1(\R^d)$.
\end{proof}

In the sequel, we fix $u_0 \in H^2(\R^d)$, and let $u(t) \in C(\R,H^2(\R^d))$ be the global solution to the continuous problem \eqref{equ-SFS} with initial data $u_0$. Consequently, we write $u_{0,h}$ as the discretization of $u_0$ as before, and solve \eqref{equ-DSFS} with initial data $u_{0,h}$. We denote by $u_h$ this solution. Note that $u_h$ is not the discretization of $u$. Then we directly compare
\begin{equation*}
    p_hu_h(t)=p_he^{-it(-\Delta_h)^2}u_{0,h}-i\lambda\int_0^t p_h e^{-i(t-s)(-\Delta_h)^2}(|u_h|^{p-1}u_h)(s)ds
\end{equation*}
with
\begin{equation*}
    u(t)=e^{-it(-\Delta)^2}u_{0}-i\lambda\int_0^te^{-i(t-s)(-\Delta)^2}(|u|^{p-1}u)(s)ds
\end{equation*}
in the space $L^2(\R^d)$. We insert several terms and write the difference as
\begin{equation*}
    \begin{aligned}
        p_hu_h(t)-u(t)
        =& ~p_he^{-it(-\Delta_h)^2}u_{0,h}-e^{-it(-\Delta)^2}u_{0}\\
        &-i\lambda \int_0^t \left(p_h e^{-i(t-s)(-\Delta_h)^2}-e^{-i(t-s)(-\Delta)^2} p_h\right)(|u_h|^{p-1}u_h)(s)ds\\
        &-i\lambda \int_0^t e^{-i(t-s)(-\Delta)^2} \left(p_h(|u_h|^{p-1}u_h)(s)-|p_hu_h|^{p-1}p_hu_h(s)\right)ds\\
        &-i\lambda \int_0^t e^{-i(t-s)(-\Delta)^2}
        \left(|p_hu_h|^{p-1}p_hu_h(s)-|u|^{p-1}u(s) \right)ds\\
        =:&~ J_1 + J_2+J_3+J_4.
    \end{aligned}
\end{equation*}

We then estimate each $J_j$ ($j=1,2,3,4$) in the following.

For $J_1$, we refer to \cite[Proposition 5.4]{HY19}(where $\alpha \in [0,1]$). Note that it can be extended directly to the case $\alpha=2$, which is
\begin{equation*}
    \begin{aligned}
        & ||p_he^{-it(-\Delta_h)^2}u_{0,h}-e^{-it(-\Delta_h)^2}u_{0}||_{L^2(\R^d)} \\
        \lesssim
        & ~h^{\frac{2}{3}}|t| \left\{||u_{0,h}||_{H^2(h\Z^d)}+||u_0||_{H^2(\R^d)}\right\} + ||p_hu_{0,h}-u_0||_{L^2(\R^d)} \\
        \lesssim 
        & ~h^{\frac{2}{3}}|t| ||u_0||_{H^2(\R^d)} + h ||u_0||_{H^2(\R^d)}\\
        \lesssim
        & ~h^{\frac{2}{3}} \langle t \rangle ||u_0||_{H^2(\R^d)}.
    \end{aligned}
\end{equation*}
Here in the second inequality we have used Lemma \ref{lem-bddness} (1) and (3), while in the third inequality we use $h^{\frac{2}{3}}>h$ when $0<h<1$. 

For $J_2$, a key ingredient is the nonlinear estimate, which states that for $p>1$, 
\begin{equation*}
    |||\nabla|^2(|u_h|^{p-1}u_h)||_{L^2(h\Z^d)} \lesssim ||u_h||_{L^{\infty}(h\Z^d)} |||\nabla|^2u_h||_{L^2(h\Z^d)}\lesssim ||u_h||_{L^{\infty}(h\Z^d)} ||u_h||_{H^2(h\Z^d)}.
\end{equation*}
Again, it can be proved by writing ($s=2$)
\begin{equation*}
    |||\nabla|^su_h||_{L^2(h\Z^d)}^2 \approx h^d \sum_{y \in h\Z^d, y \neq 0} \frac{||u_h(\cdot+y)-u_h(\cdot)||^2_{L^2(h\Z^d)}}{|y|^{d+2s}}
\end{equation*}
and following the same argument in \cite[Lemma 4.3]{HY19}. Therefore, we have
\begin{equation*}
    \begin{aligned}
        ||J_2||_{L^2(\R^d)} 
        & \lesssim h^{\frac{2}{3}}|t| \int_{0}^{t} |||u_h|^{p-1}u_h(s)||_{H^2(h\Z^d)} ds \\
        & \lesssim h^{\frac{2}{3}} |t|^{2}||u_h||^{p-1}_{L_s^{\infty}([0,t];L^{\infty}(h\Z^d))}||u_h||_{C_t([0,t];H^2(h\Z^d))}\\
        & \lesssim h^{\frac{2}{3}} |t|^{2}||u_{0,h}||^p_{H^2(h\Z^d)} \lesssim h^{\frac{2}{3}} |t|^{2}||u_0||^p_{H^2(\R^d)}.
    \end{aligned}
\end{equation*}

For $J_3$, we begin with verifying the almost distribution law for linear interpolation, i.e.
\begin{equation*}
    ||p_h(|u_h|^{p-1}u_h)-|p_hu_h|^{p-1}p_hu_h||_{L^2(\R^d)} \lesssim h^{\frac{2}{3}}||u_h||^{p-1}_{L^{\infty}(h\Z^d)}||u_h||_{H^2(h\Z^d)}.
\end{equation*}
It can be proved by applying \cite[Proposition 5.8]{HY19} (with $\alpha=1$) and trivial embedding results. Hence, we get
\begin{equation*}
    ||J_3||_{L^2(\R^d)} \lesssim h^{\frac{2}{3}} \int_{0}^{t}||u_h(s)||^{p-1}_{L^{\infty(h\Z^d)}}||u_h(s)||_{H^2(h\Z^d)}ds
\end{equation*}
and only need to repeat the argument for $J_2$ to get
\begin{equation*}
    ||J_3||_{L^2(\R^d)} \lesssim h^{\frac{2}{3}} |t| ||u_0||_{H^2(\R^d)}.
\end{equation*}

For $J_4$, we simply observe the fundamental theorem of calculus, i.e. 
\begin{equation*}
    \begin{aligned}
        |u|^{p-1}u-|v|^{p-1}v
        & = \int_{0}^{1} \frac{d}{ds}(|su+(1-s)v|^{p-1}(su+(1-s)v))ds \\
        & = \frac{p+1}{2}\left\{\int_{0}^{1}|su+(1-s)v|^{p-1}ds\right\}(u-v)\\
        & + \frac{p-1}{2}\left\{\int_{0}^{1}|su+(1-s)v|^{p-3}(su+(1-s)v)^2ds\right\}\overline{u-v}.
    \end{aligned}
\end{equation*}
Inserting this into $J_4$ (with $u$ and $p_hu_h$)  gives 
\begin{equation*}
    \begin{aligned}
        ||J_4||_{L^2(\R^d)} 
        & \lesssim \int_{0}^{t} \left(||p_hu_h(s)||_{L^{\infty}(\R^d)}+||u(s)||_{L^{\infty}(\R^d)}\right)^{p-1}||p_hu_h(s)-u(s)||_{L^2(\R^d)}ds\\
        & \lesssim \int_{0}^{t}\left(||u_h(s)||_{L^{\infty}(h\Z^d)}+||u(s)||_{L^{\infty}(\R^d)}\right)^{p-1}||p_hu_h(s)-u(s)||_{L^2(\R^d)}ds.
    \end{aligned}    
\end{equation*}
It is obvious that $||p_hf||_{L^{\infty}(\R^d)} \lesssim ||f||_{L^{\infty}(h\Z^d)}$ by the definition of $p_h$.

Therefore, when $d=2$, collecting all above estimates and applying the Gronwall's inequality, we obtain
\begin{equation*}
    \begin{aligned}
        &||p_hu_h(t)-u(t)||_{L^2(\R^2)} \\
         \lesssim & ~h^{\frac{2}{3}} \langle t \rangle^2 (1+||u_0||_{H^2(\R^2)})^p\\
        & +\int_{0}^{t}\left(||u_h(s)||_{L^{\infty}(h\Z^2)}+||u(s)||_{L^{\infty}(\R^2)}\right)^{p-1}||p_hu_h(s)-u(s)||_{L^2(\R^2)}\\
        \lesssim &~ h^{\frac{2}{3}} \langle t \rangle^2 (1+||u_0||_{H^2(\R^2)})^p \exp \left\{ \int_{0}^{t}\left(||u_h(s)||_{L^{\infty}(h\Z^2)}+||u(s)||_{L^{\infty}(\R^2)}\right)^{p-1}\right\}ds \\
        \lesssim &~ h^{\frac{2}{3}} \langle t \rangle^2 (1+||u_0||_{H^2(\R^2)})^p ~e^{Bt} \\
        \lesssim &~ h^{\frac{2}{3}} e^{(B+1)|t|} (1+||u_0||_{H^2(\R^2)})^p. 
    \end{aligned}
\end{equation*}
In the third inequality, we use the time-averaged $L^{\infty}(\R^2)$ (Theorem \ref{thm-global of DSDF}) and $L^{\infty}(h\Z^2)$ (Theorem \ref{thm-global of SDF}) bounds on $u(s)$ and $u_h(s)$. We complete the proof.

\section*{Acknowledgement}

The authors wish to express gratitude to Prof. Changxing Miao and Jiqiang Zheng for their helpful explanations and to Cheng Bi for useful discussions. B.H. is supported by NSFC, No. 12371056, and by Shanghai Science and Technology Program [Project No. 22JC1400100].


\printbibliography

\end{document}